\documentclass[letterpaper, 10 pt, conference,onecolumn]{ieeeconf}

\IEEEoverridecommandlockouts 

\usepackage[latin1]{inputenc}
\usepackage{graphicx}
\usepackage{amssymb}
\usepackage{amsmath}
\usepackage{psfrag}
\usepackage{xspace} 
\usepackage{color}
\usepackage{multirow,tabularx}
\usepackage{algorithm}
\usepackage{algorithmicx}
\usepackage{algpseudocode}
\usepackage{cases}
\usepackage{balance}

\newtheorem{lemma}{Lemma}
\newtheorem{theorem}{Theorem}

\newtheorem{assumption}{Assumption}

\newtheorem{remark}{Remark}

\def\rr{\mathbb{R}}

\def\cM{\mathcal{M}}

\def\cL{\mathcal{L}}

\def\cD{\mathcal{D}}

\def\cV{\mathcal{V}}
\def\cC{\mathcal{C}}
\def\cE{\mathcal{E}}

\def\eps{\varepsilon}

\newcommand{\ds}{\displaystyle}

\renewcommand{\baselinestretch}{1.5}

\title{\vspace*{-4mm}Cooperative versus decentralized strategies\\in three-pursuer single-evader games}

\author{Marco Casini, Andrea Garulli\thanks{M.~Casini and A.~Garulli are with the Dipartimento di
		Ingegneria dell'Informazione e Scienze Matematiche, Universit\`a di Siena, via Roma~56, 53100 Siena, Italy. E-mail:casini@diism.unisi.it,~garulli@diism.unisi.it.}}

\date{}

\begin{document}
\maketitle \thispagestyle{empty} \pagestyle{empty}
	
\begin{abstract}
The value of cooperation in pursuit-evasion games is investigated. The considered setting is that of three pursuers chasing one evader in a planar environment. The optimal evader trajectory for a well-known decentralized pursuer strategy is characterized. This result is instrumental to derive upper and lower bounds to the game length, in the case in which the pursuers cooperate in the chasing strategy. It is shown that the cooperation cannot reduce the capture time by more than one half with respect to the decentralized case, and that such bound is tight. 
\end{abstract}

\begin{keywords}
	pursuit-evasion games, autonomous agents, cooperative control, differential games.
\end{keywords}	

\section{Introduction}\label{sec:intro}

Pursuit-evasion games involving many pursuers against a single evader have been intensively investigated for quite a long time.  
These problems fall within the broader class of differential games \cite{isaacs1965,Petrosyan1993}, for which a vast literature is available (see \cite{chernousko1985} and references therein). A distinctive trait of these games is that the union of the reachable sets of all pursuers is in general a nonconvex, or even disconnected, set. This makes the characterization of successful pursuit strategies nontrivial, even for simple problem settings, like the case of three pursuers chasing one evader in the plane \cite{pashkov1995}. More complex environments usually lead to very involved geometric conditions \cite{alexander2009}.

Recent years have witnessed a renewed interest towards cooperative capture games, motivated both by the variety of applications in which they play a key role and by the development of technologies allowing networks of agents to collaborate in the execution of complex tasks, see, e.g., \cite{vidal2002,bopardikar2008,chung2011,noori2016,Makkapati2018}.
An interesting aspect that has not been fully understood yet, concerns the benefits of the pursuers to cooperate in the capture game. A popular strategy in the literature is the one proposed in \cite{kopparty2005}, which can be seen as a specialization of classical differential game solutions, proposed for example in \cite{pshenichnyi1976,pshenichnyi1981}. Such strategies are fully decentralized, as each pursuer is aware only of the evader move but does not have any information about the other pursuers' behavior.
On the other hand, cooperative strategies in which pursuers have full or partial knowledge of the game state, have been recently proposed by several authors \cite{zhou2016,Ramana2017,kothari2017,Ramana2017b}. However, it is not easy to quantify the advantage of playing such strategies with respect to the fully decentralized ones, which do not require the use of a complex communication infrastructure.

In this paper, we consider the problem of three pursuers chasing one evader in a planar environment. The first contribution is a strategy guaranteeing the maximum survival time for the evader, when the pursuers adopt the decentralized strategy proposed in \cite{kopparty2005}. Then, upper and lower bounds for the game length are derived, assuming that the pursuers play in a cooperative way. The main result shows that, no matter how they collaborate, the pursuers cannot reduce the game length to less than one-half of that resulting from the decentralized strategy. Moreover, it is shown that the above bounds are tight, in the sense that there exist games in which the bounds are actually achieved. Finally, the maximum advantage deriving from cooperation is quantified for some specific game initial conditions.

The paper is structured as follows. In Section~\ref{sec:PEG}, the three-pursuer one-evader game is formulated. In Section~\ref{sec:decentralized}, the decentralized pursuit strategy proposed in \cite{kopparty2005} is recalled and the corresponding optimal evader's strategy is devised along with the resulting game length. In Section~\ref{sec:lower_bound}, a lower bound on capture time is derived, irrespectively of the strategy adopted by the pursuers. A comparison between the game duration for cooperative and decentralized pursuers' strategies is reported in Section~\ref{sec:strategy_comparison} while some examples of games are reported in Section~\ref{sec:example}. Conclusions and future developments are reported in Section~\ref{sec:conclusions}.

\section{Pursuit-evasion game}\label{sec:PEG}

\subsection{Notation}\label{sec:notation}
Let $\rr^n$ be the $n$-dimensional Euclidean space and $\|\cdot\|$ be the Euclidean norm. The transpose of a vector $v$ is denoted by $v'$. Let $V,W\in\rr^2$, we denote by $\overline{VW}$ the segment with $V$ and $W$ as endpoints. Let $V,v\in\rr^2$, we denote by $\cL(V,v)$ the line passing through $V$ with direction $v$, i.e.,
$$
\cL(V,v)=\{X\in\rr^2\colon X=V+\alpha v,\,\alpha\in\rr\}.
$$

\subsection{Problem formulation}
A pursuit-evasion game involving three pursuers is considered. It is assumed that the players move in an open and empty two-dimensional environment. Let $E(t)\in\rr^2$ and $P_i(t)\in\rr^2,~i=1,2,3,$ denote the evader and pursuers location at time $t$, respectively. The aim of the pursuers is to capture the evader, i.e., $P_i(t)=E(t)$ for at least one $i\in\{1,2,3\}$, at some time $t$.

The following assumptions are enforced throughout the text.
\begin{assumption}\label{ass:same_speed}
	The pursuers and the evader have the same speed, set to 1 without loss of generality. Moreover, we assume the players have simple motion, i.e., they can freely move in any direction.
\end{assumption}
\begin{assumption}\label{ass:inside_hull}
	The initial evader position is strictly inside the convex hull of the pursuers.
\end{assumption}
If the players move at the same speed, enforcing Assumption~\ref{ass:inside_hull} is standard, otherwise the evader may easily escape going straight along a direction opposite to the convex hull of the pursuers \cite{Jankovic1978}. On the contrary, if Assumption~\ref{ass:inside_hull} holds, there exist pursuers' strategies which guarantee capture of the evader in finite time \cite{Jankovic1978, kopparty2005}.

For a given configuration of the players at time $t$, let us define as $\cV(t)$ the Voronoi cell associated to the evader, i.e., the region of the plane closer to the evader than to the pursuers, at time $t$. Under Assumptions~\ref{ass:same_speed}-\ref{ass:inside_hull}, $\cV(t)$ turns out to be a triangle; let us denote by $V_i(t),~i=1,2,3,$ its vertices. For a given triangle $\cV$, we denote by $l,m,s$ the longest, medium and smallest edge of $\cV$, respectively. Moreover, we name the vertices of $\cV$ such that $V_1$ is the vertex joining	 the longest and medium edges, while pursuers are labeled such that $P_i$ is the pursuer farthest from $V_i$. It can be easily observed that (see Fig.~\ref{fig:D_strategy})
\begin{equation}\label{eq:P_lontani}
\|V_i-P_i\|>\|V_i-E\|~,~ i=1,2,3\ .
\end{equation}

The pursuit game can be played in continuous time, in which players simultaneously set their velocities, or discrete time, when players move in turn. In the latter case, by convention, the evader moves first, and then all pursuers move simultaneously after they have observed the evader's move. Let $e(t)\in\rr^2$ and $w_i(t)\in\rr^2$, $i=1,2,3,$ be such that $\|e(t)\|=1$ and $\|w_i(t)\|=1$.

In continuous-time games, the motion models are
\begin{equation}\label{eq:ct}
\left\{
\begin{array}{l}
\dot{E}(t)=e(t)\\
\dot{P}_i(t)=w_i(t)~,~i=1,2,3
\end{array}
\right.
\end{equation}
while in discrete-time games they move following
\begin{equation}\label{eq:dt}
\left\{
\begin{array}{l}
E(t+1)=E(t)+e(t)\\
P_i(t+1)=P_i(t)+w_i(t)~,~i=1,2,3
\end{array}
\right.\ .
\end{equation}

In this paper, two classes of pursuers' strategies are considered: \emph{decentralized} (uncoordinated) and \emph{cooperative} (coordinated). A pursuers' strategy is said decentralized if each pursuer does not have information about the other pursuers and computes its move solely on the base of the evader position and move (i.e., $E(t)$ and $e(t)$) and its own state $P_i(t)$. On the contrary, in the cooperative case, each pursuer knows the full game state and chooses its move depending on $P_1(t),P_2(t),P_3(t),E(t)$ and $e(t)$. In both cases, we assume the evader has a complete knowledge of the state of the game.

\section{Decentralized pursuit strategy}\label{sec:decentralized}

In this section, a decentralized pursuers' strategy is recalled. It has been proposed in \cite{pshenichnyi1976} for the continuous-time framework, and in \cite{kopparty2005} (under the name ``Planes'') within the discrete-time setting.
Such a strategy is designed in $\rr^n$ and it guarantees capture in finite time under Assumptions~\mbox{\ref{ass:same_speed}-\ref{ass:inside_hull}}.
In this paper, we will restrict the analysis to the two-dimensional space. 

Let $C_i(t)=(P_i(t)+E(t))/2$ and $z_i(t)=P_i(t)-E(t)$, $i=1,2,3$. Denote by $z_i(t)^\perp$ a vector orthogonal to $z_i(t)$.
Define
$$
B_i(t)=\cL(E(t),e(t))\cap\cL(C_i(t),z_i(t)^\perp)\ .
$$
For both models \eqref{eq:ct} and \eqref{eq:dt}, define the pursuers' moves $w_i(t)$, $i=1,2,3,$ as follows (see Fig.~\ref{fig:D_strategy})
\begin{subnumcases}{w_i(t)=}
~~~~~~~e(t)  & if  $z_i(t)'e(t)\le0$\label{decentralized_strategy_a}\\[1mm]
\frac{B_i(t)-P_i(t)}{\|B_i(t)-P_i(t)\|}   & if  $z_i(t)'e(t)>0$\label{decentralized_strategy_b}
\end{subnumcases}

\begin{figure}[htb]
	\centering %
	\includegraphics[width=.5\columnwidth]{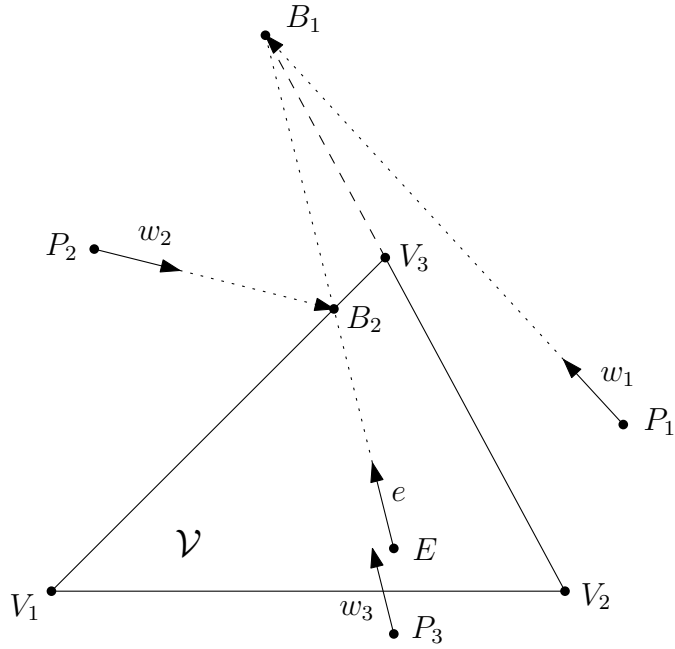}
	\caption{Example of pursuers' moves based on the $\cD$-strategy. Pursuers $P_1$ and $P_2$ follow \eqref{decentralized_strategy_b}, while $P_3$ obeys to \eqref{decentralized_strategy_a}.} \label{fig:D_strategy}
\end{figure}

In words, when the evader moves towards an edge of the Voronoi cell (condition \eqref{decentralized_strategy_b}), the corresponding pursuer makes a specular move which leaves the edge unchanged. Conversely, if the evader moves away from that edge (condition \eqref{decentralized_strategy_a}), the pursuer makes the same move of the evader, thus causing a shrinking of the Voronoi cell (in Fig.~\ref{fig:D_strategy}, this occurs for pursuer $P_3$). It is worth remarking that in the latter case, the direction of the edge of the Voronoi cell does not change. Therefore, all the Voronoi cells throughout the entire game are similar triangles, no matter of the path followed by the evader.

We refer to the decentralized pursuers' strategy in \eqref{decentralized_strategy_a}-\eqref{decentralized_strategy_b} as $\cD$-strategy. Notice that in \eqref{decentralized_strategy_a}-\eqref{decentralized_strategy_b} each pursuer's move depends solely on the evader and its own state. 

The results presented in this paper hold for both the continuous-time and the discrete-time framework. However, in the discrete-time approach, such results do not take into account the quantization effect introduced by the discrete move, see e.g., \cite{sgall01,casini_garulli_LCSS17}. Therefore, for ease of exposition, in the sequel we will refer to the continuous-time setting.

\subsection{Optimal evasion strategy against the $\cD$-strategy}

Hereafter, a pursuers' strategy will be said optimal when it guarantees capture in minimum time, while an evader's strategy is optimal if it guarantees survival of the evader for the longest time.
In this subsection, an optimal evader's strategy is devised for games in which the pursuers play the $\cD$-strategy. Let us name such strategy as $\cE$. It is worthwhile to notice that there exist several evader strategies which lead to the same optimal capture time; we will just focus on one of them.

Let $E(0)$, $P_i(0)$, $i=1,2,3,$ be given, and let $\cV(0)$ be the corresponding Voronoi cell. Without loss of generality let us assume the longest edge of $\cV(0)$ be $l=\|V_1(0)-V_2(0)\|$. Let $v=V_1(0)-V_2(0)$ and denote by $S$ and $Q$ the intersection points between the line passing through $E(0)$ parallel to the longest side of the triangle and $\cV(0)$ (see Fig.~\ref{fig:M_D_steps}), i.e.,
\begin{equation}\label{eq:S}
S=\overline{V_2(0) V_3(0)}\cap\cL(E(0),v)
\end{equation}
and
\begin{equation}\label{eq:Q}
Q=\overline{V_1(0) V_3(0)}\cap\cL(E(0),v)\ .
\end{equation}

Define the unitary vectors connecting the vertices of $\cV(0)$ as
\begin{equation}\label{eq:v_i}
v_{ij}=\frac{V_i(0)-V_j(0)}{\|V_i(0)-V_j(0)\|}\quad,\quad i\neq j\ .
\end{equation}
In a similar way, define
\begin{align}
&v_{SE}=\frac{S-E(0)}{\|S-E(0)\|}~,~v_{QE}=\frac{Q-E(0)}{\|Q-E(0)\|}\nonumber\\
&v_{V_1 Q}=\frac{V_1(0)-Q}{\|V_1(0)-Q\|}\ .\nonumber
\end{align}

Let us now define the evader's strategy $\cE$ as follows, see Fig.~\ref{fig:M_D_steps}.
\begin{enumerate}
	\item From $E(0)$ the evader moves along $v_{QE}$ to $Q$.
	\item Once in $Q$, it moves along $v_{V_1 Q}$ to $V_1(0)$.
	\item Once in $V_1(0)$, it moves towards $V_2(0)$, until it reaches the farthest vertex of the current Voronoi cell, where it is captured.
\end{enumerate}

\begin{remark}
	In order to simplify the exposition, when it is stated that the evader moves to a point which lies on the boundary of $\cV$, we actually mean that it moves to an interior point of $\cV$ which is arbitrarily close to the boundary. 
	In fact, such a move is feasible and safe, due to the fact that the evader can reach any point inside $\cV$ without being captured, by definition of the Voronoi cell. 
	For instance, referring to item 1) of the $\cE$-strategy, the evader will move along $v_{QE}$ to a point $\widetilde Q$ such that $\|\widetilde Q-Q\|<\delta$, for a small $\delta>0$. 
	In this respect, all the results presented in the paper must be intended as limit results, obtained by letting $\delta$ tend to 0.
\end{remark}

\begin{figure}[htb]
	\centering %
	\includegraphics[width=0.8\columnwidth]{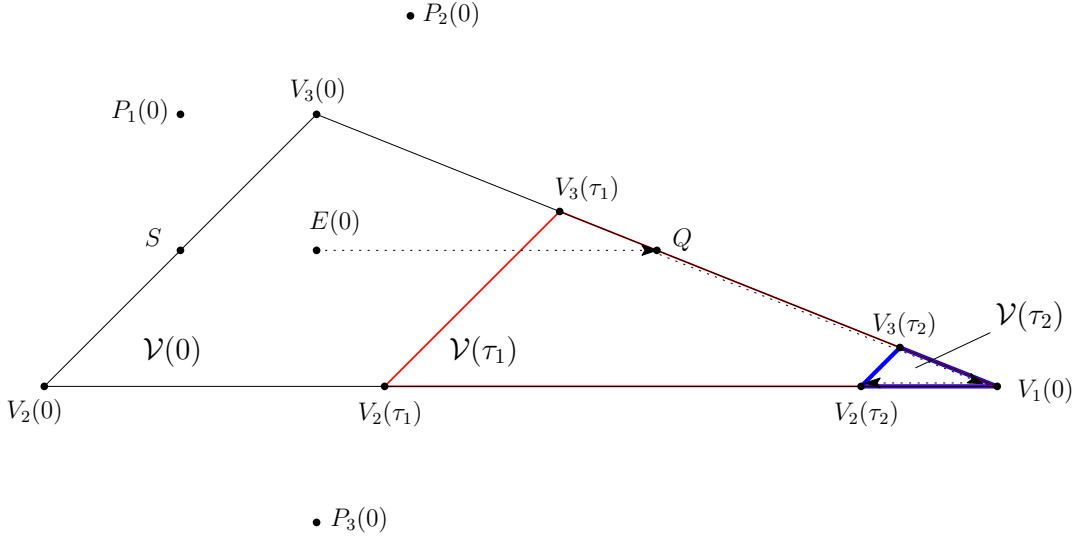}
	\caption{Sketch of the evader's strategy $\cE$. The evader follows the dotted arrows. The Voronoi cells at the beginning of steps 1, 2 and 3 are depicted in black, red and blue, respectively. } \label{fig:M_D_steps}
\end{figure}

In the next theorem, it is proved that the $\cE$-strategy is optimal for the evader when the pursuers play the $\cD$-strategy, and the related game length $M_\cD$ is given. 

\begin{theorem}\label{th:mosse_decentralizzate}
	Let the pursuers play the $\cD$-strategy and the evader play the $\cE$-strategy. Then, the game will last for a time
	\begin{equation}\label{eq:M_D}
	M_\cD=\|S-Q\|+\|Q-V_1(0)\|\ .
	\end{equation}
	Moreover, the $\cE$-strategy is optimal for the evader.
\end{theorem}
\begin{bf}
	See appendix.
\end{bf}

\begin{remark}
	In \cite{pshenichnyi1976}, an upper bound on the capture time is reported for the generic game played in $\rr^n$ involving $m$ pursuers. Let $z_i=P_i-E$, by fixing $n=2$ and $m=3$, such bound turns out to be
	$$
	B_P=\ds\frac{\ds\max_{i=1,2,3}\|z_i\|}{\delta_0}
	$$
	with
	$$
	\delta_0=\min_{\|p\|=1} \max_{i=1,2,3} \frac{p'z_i}{\|z_i\|}\ .
	$$
	It can be shown that $B_P$ is in general much larger than the exact number of moves given by Theorem~\ref{th:mosse_decentralizzate}. For instance, for $10^6$ randomly generated games, the ratio between $B_P$ and $M_\cD$ turned out ranging from about 1.2 to over 3000.
\end{remark}

\section{Lower bound on capture time}\label{sec:lower_bound}
Assume all players have complete information about the state of the game. The following theorem gives a lower bound $\underline B$ on the capture time, i.e., the evader, playing a suitable strategy, may avoid capture for at least a time $\underline B$, for any possible pursuers' strategy.

\begin{theorem}\label{th:LB_general}
	Let
	$$
	i^*=\arg\max_{i=1,2,3} \frac{1}{2}\left(\|V_i(0)-P_i(0)\|+\|V_i(0)-E(0)\|\right)
	$$
	and define $V^*(t)=V_{i^*}(t)$ and $P^*(t)=P_{i^*}(t)$, for all $t$.
	
	Then, the evader is able to survive for at least a time $\underline B$, where
	\begin{align}
	\underline B&=\frac{1}{2}\left(\|V^*(0)-P^*(0)\|+\|V^*(0)-E(0)\|\right)\label{eq:LB_general3}\\
	&=\max_{i=1,2,3} \frac{1}{2}\left(\|V_i(0)-P_i(0)\|+\|V_i(0)-E(0)\|\right)\label{eq:LB_general2}\\
	&=\max_{i=1,2,3} (\|V_i(0)-E(0)\|\nonumber\\
	&~~~~~~+\frac{1}{2}\left(\|V_i(0)-P_i(0)\|-\|V_i(0)-E(0)\|\right))\ .\label{eq:LB_general1}
	\end{align}	
	~\vspace*{-10pt}
\end{theorem}
\begin{proof}
	Let us consider the following evader's strategy. From its initial position $E(0)$, it moves straight to $V_i(0)$ for a time $\tau_1=\|E(0)-V_i(0)\|$ for a given $i\in\{1,2,3\}$. Since $V_i(0)\in\cV(0)$, it can be arbitrarily approached, irrespectively of the pursuers' strategy.
	
	Let $d_0=\|P_i(0)-V_i(0)\|$. By \eqref{eq:P_lontani}, one has 
	\begin{equation}\label{eq:dist_P}
	d_0=\|P_i(0)-V_i(0)\|> \|E(0)-V_i(0)\|=\tau_1.
	\end{equation}
	Since the speed of the pursuers is set to 1, by \eqref{eq:dist_P} the distance between the evader and the pursuer $P_i$ at time $\tau_1$ is such that
	\begin{align}
	\|P_i(\tau_1)-E(\tau_1)\|&=\|P_i(\tau_1)-V_i(0)\|\nonumber\\
	&\ge \|P_i(0)-V_i(0)\|-\tau_1=d_0-\tau_1\!>\!0.\nonumber
	\end{align}
	Define $d_{\tau_1}=d_0-\tau_1>0.$
	Let $Z=(E(\tau_1)+P_i(\tau_1))/2$ be the midpoint between $E(\tau_1)$ and $P_i(\tau_1)$. By the definition of Voronoi cell, $Z$ lies on the boundary of $\cV(\tau_1)$. Assuming the evader goes straight to $Z$, it covers a distance
	$$
	\|Z-E(\tau_1)\|=\frac{\|P_i(\tau_1)-E(\tau_1)\|}{2}\ge \frac{d_{\tau_1}}{2}
	$$ 
	and then it is captured in $Z$. Hence, the time needed to cover the entire path turns out to be
	\begin{align}
	T&\ge \tau_1+\frac{d_{\tau_1}}{2}=\tau_1+\frac{d_0-\tau_1}{2}=\frac{1}{2}(\tau_1+d_0)\nonumber\\
	&=\frac{1}{2} (\|V_i(0)-E(0)\|+\|V_i(0)-P_i(0)\|)\ .\label{eq:c1}
	\end{align}
	Therefore, the right hand side of \eqref{eq:c1} is a lower bound to the evader's survival time. By taking the maximum with respect to $i=1,2,3$, one gets the lower bound $\underline{B}$ in (10). The expressions (11) and (12) follow from straightforward manipulations. 
\end{proof}

\section{Cooperative vs decentralized pursuit strategies}\label{sec:strategy_comparison}

The aim of this section is to analyze the potential advantage of the pursuers to cooperate in the pursuit task, with respect to adopting the $\cD$-strategy discussed in Section III.
The following lemmas are instrumental to prove the main results. Hereafter, the time dependence is omitted when it is clear from the context. 

\begin{lemma}
	Let the pursuers play the $\cD$-strategy and let $l$ denote the longest edge of $\cV$. Then
	\begin{equation}\label{eq:MD_le_l}
	M_\cD\le l\ .
	\end{equation}
\end{lemma}
\begin{proof}
	Let $S$ and $Q$ be defined as in \eqref{eq:S}-\eqref{eq:Q} and assume $l=\|V_1-V_2\|$, see Fig.~\ref{fig:M_D}. It holds
	\begin{align}
	\|V_1&-V_3\| / \|V_1-V_2\|=\|Q-V_3\| / \|Q-S\|\nonumber\\
	&=(\|V_1-V_3\|-\|Q-V_3\|) / (\|V_1-V_2\|-\|Q-S\|).\nonumber
	\end{align}
	Since $\|V_1-V_3\|\le\|V_1-V_2\|$ one has
	$$
	\|V_1-V_3\|-\|Q-V_3\|\le\|V_1-V_2\|-\|Q-S\|
	$$
	or equivalently
	$$
	\|Q-V_1\|\le \|V_1-V_2\|-\|Q-S\|\ .
	$$
	So, by Theorem~\ref{th:mosse_decentralizzate},
	\begin{align}\nonumber
	M_\cD&=\|Q-S\|+\|Q-V_1\|\\
	&\le \|Q-S\|+\|V_1-V_2\|-\|Q-S\|=l\ .\nonumber
	\end{align}
\end{proof}
\begin{figure}[htb]
	\centering %
	\includegraphics[width=0.8\columnwidth]{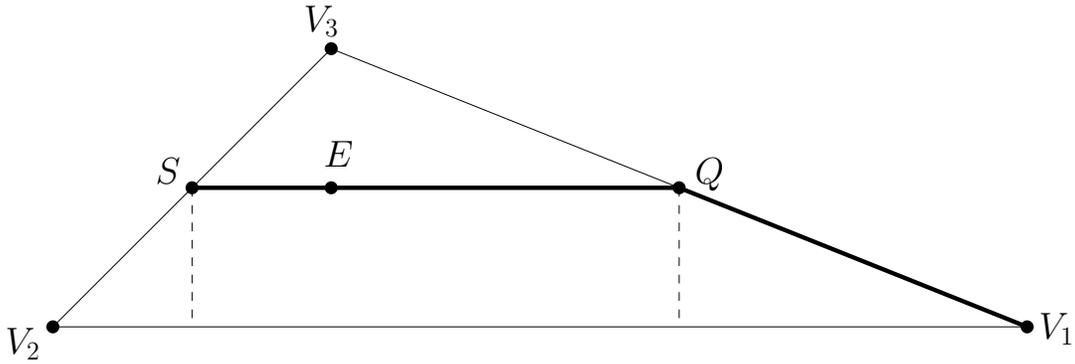}
	\caption{Game length when pursuer is playing the $\cD$-strategy. $M_\cD$ is the length of the bold line.} \label{fig:M_D}
\end{figure}

\begin{lemma}
	Let $\underline B$ be given as in Theorem~\ref{th:LB_general} and let $l$ denote the largest edge of $\cV$. Then,
	\begin{equation}\label{eq:B_ge_l_mezzi}
	\underline B\ge l/2\ .
	\end{equation}
\end{lemma}
\begin{proof}
	Let $l=\|V_1-V_2\|$. By the triangle inequality,
	\begin{align}
	2\max\{\|E-V_1\|,\|E-V_2\|\}&\ge \|E-V_1\|+\|E-V_2\|\nonumber\\
	&\ge \|V_1-V_2\|=l\ .\nonumber
	\end{align}
	By \eqref{eq:P_lontani} and \eqref{eq:LB_general1}, one has
	$$
	\underline B\ge \max_{i=1,\ldots,3} \|V_i-E\|\ge \max\{\|E-V_1\|,\|E-V_2\|\}\ge l/2\ .
	$$
\end{proof}

For a given pursuit-evasion game, let us denote by $\cC$ the optimal cooperative pursuers' strategy and by $M_\cC$ the related maximum capture time, i.e., the time at which capture occurs if the evader plays at its best.
The following theorem reports the relation between $M_\cC$ and the decentralized game length $M_\cD$ provided by Theorem~\ref{th:mosse_decentralizzate}.

\begin{theorem}\label{th:main}
	Let $M_\cD$ be given by \eqref{eq:M_D} and $M_\cC$ be the optimal game length in a cooperative pursuers' setting. Then,
	\begin{equation}\label{eq:main_result}
	M_\cC\le M_\cD\le 2M_\cC\ .
	\end{equation}
\end{theorem}
\begin{proof}
	Since the $\cC$-strategy is optimal among all pursuit strategies, it cannot be worse than the $\cD$-strategy, and hence $M_\cC\le M_\cD$ trivially holds.
	By \eqref{eq:MD_le_l} and \eqref{eq:B_ge_l_mezzi}, one has
	$$
	M_\cD\le l\le 2\underline B\le 2 M_\cC
	$$
	where the last inequality comes from the fact that $\underline B$ is a lower bound on the game length for any pursuers' strategy.
\end{proof}
By Theorem~\ref{th:main}, one immediately has
\begin{equation}\label{eq:disuguaglianze}
\underline B\le M_\cC\le M_\cD\le 2\underline B\le 2 M_\cC\ .
\end{equation}

In the sequel, we prove that there exist games in which
\begin{equation}\label{eq:C_uguale_D}
M_\cC=M_\cD
\end{equation}
and other in which
\begin{equation}\label{eq:C_uguale_2D}
M_\cC=\frac{1}{2}M_\cD
\end{equation}
this meaning that both bounds in \eqref{eq:main_result} are tight.
\begin{theorem}\label{th:tight_sx}
	There exist games such that $M_\cC=M_\cD$.
\end{theorem}
\begin{proof}
	To prove the theorem, we show that there exist games for which $M_\cD=\underline B$. In fact, by \eqref{eq:disuguaglianze}, $M_\cD=\underline B$ implies $M_\cC=M_\cD$. Let us choose a game initial condition such that $\cV$ is a right triangle and let us adopt the notation shown in Fig.~\ref{fig:right_triangle}. Since $\overline{SQ}$ is the hypotenuses of the triangle with vertices $S$, $Q$ and $V_3$, by Theorem~\ref{th:mosse_decentralizzate}, one easily gets $M_\cD\ge\|V_1-V_3\|=m$. Moreover, by \eqref{eq:MD_le_l}, it holds
	$$
	m\le M_\cD\le l\ .
	$$
	Let the smallest edge $s$ shrink to 0. One has
	$$
	\lim_{s\to0}l=\lim_{s\to0}\sqrt{m^2+s^2}=m
	$$
	and hence
	\begin{equation}\label{eq:sx1}
	\lim_{s\to0}M_\cD=m\ . 
	\end{equation}
	Moreover, it is easy to show that as $s\to0$, $P_1,E$ and $V_1$ tend to be collinear. This implies that
	$$
	\lim_{s\to0}\left(\|V_1-P_1\|-\|V_1-E\|\right)=\lim_{s\to0}\|P_1-E\|\ .
	$$
	By \eqref{eq:LB_general1}, one has
	\begin{align}
	\lim_{s\to0}\underline B&\ge \lim_{s\to0}\left(\|V_1-E\|+\frac{1}{2}(\|V_1-P_1\|-\|V_1-E\|)\right)\nonumber\\
	&= \lim_{s\to0}\left(\|V_1-E\|+\frac{1}{2}\|P_1-E\|\right)=m\ .\label{eq:sx2}
	\end{align}
	Since $M_\cD\ge\underline B$, by \eqref{eq:sx1} and \eqref{eq:sx2} one has
	$$
	\lim_{s\to0} M_\cD= \lim_{s\to0} \underline B=m
	$$
	which concludes the proof.
	\begin{figure}[tb]
		\centering %
		\includegraphics[width=0.8\columnwidth]{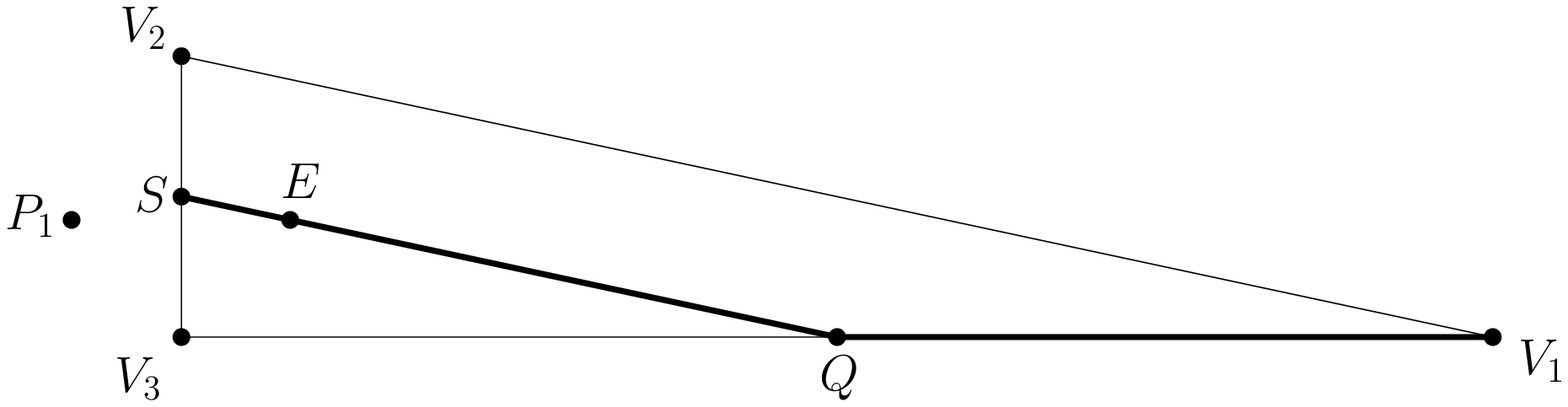}
		\caption{As $\|V_2-V_3\|\to0$, one has $M_\cD\to\underline B$ and hence $M_\cD=M_\cC$. The length of the bold line is equal to $M_\cD$.} \label{fig:right_triangle}
	\end{figure}
	\begin{figure*}[!t]
		\centering %
		\includegraphics[width=.95\columnwidth]{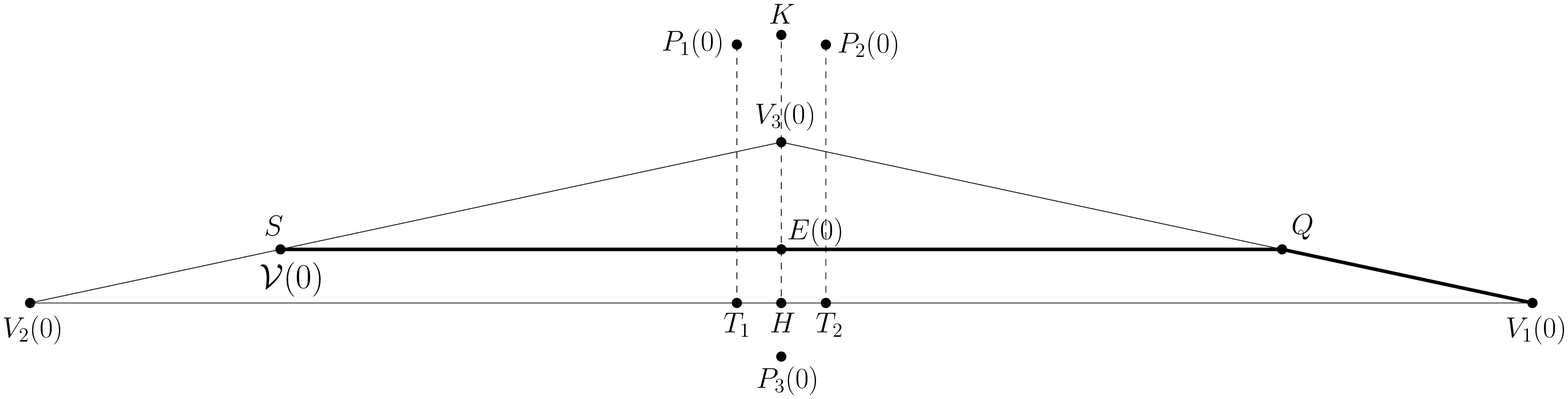}
		\caption{Voronoi cell $\cV$ at time 0. As $\|V_3-H\|\to0$, one has $M_\cD\to 2 M_\cC$. The length of the bold line is equal to $M_\cD$.} \label{fig:fat_triangle}
	\end{figure*}
	\begin{figure*}[!t]
		\centering %
		\includegraphics[width=.95\columnwidth]{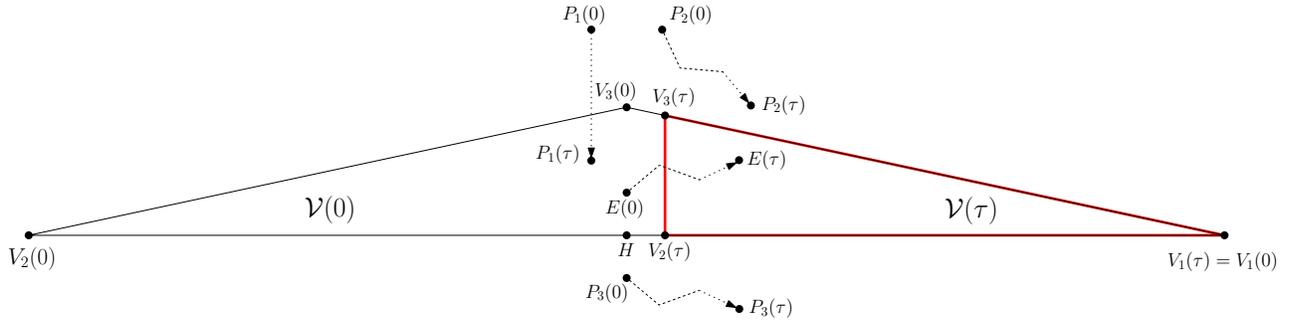}
		\caption{First step of strategy $\widehat\cC$. Voronoi cell $\cV$ at time 0 (black) and at time $\tau$ (red).} \label{fig:right_triangle_bis}
	\end{figure*}
\end{proof}

\begin{theorem}\label{th:tight_dx}
	There exists games such that $M_\cD=2 M_\cC$.
\end{theorem}
\begin{proof}
	We show that there exist games in which $M_\cC=\underline B$ and $M_\cD=2\underline B$. Let $\cV(0)$ be the isosceles triangle depicted in Fig.~\ref{fig:fat_triangle}. Define $H,T_1,T_2$ the projections of $V_3(0),P_1(0),P_2(0)$ on $\overline{V_1(0) V_2(0)}$, respectively. Let the evader initially lie on the segment $\overline{V_3(0) H}$, define $\eps=\|V_3(0)-H\|$ and $K=E(0)+2(V_3(0)-E(0))$. Since
	$\|V_3(0)-P_1(0)\|=\|V_3(0)-E\|=\|V_3(0)-K\|$ one has
	\begin{align}\label{eq:2eps}
	\|P_1(0)-T_1\|&\le\|P_1(0)-V_3(0)\|+\|V_3(0)-H\|\nonumber\\
	&=\|V_3(0)-K\|+\|V_3(0)-H\|\nonumber\\
	&\le 2 \|V_3(0)-H\|=2\eps\ .
	\end{align}
	Now, let $\eps$ tend to 0. In Fig.~\ref{fig:fat_triangle}, $M_\cD$ corresponds to the length of the bold line. By following a similar reasoning as in the proof of Theorem~\ref{th:tight_sx}, one has
	\begin{equation}\label{eq:MDproof}
	\lim_{\eps\to0} M_\cD=\lim_{\eps\to0}\|S-Q\|+\|Q-V_1(0)\|=
	\|V_2(0)-V_1(0)\|\ .
	\end{equation}
	
	Let us now introduce a two-step cooperative pursuers' strategy, denoted by $\widehat\cC$. Let $v_{QE}=Q-E(0)$. Assume first that the evader keeps moving in the half-plane $v_{QE}'e(t)\ge0$ for a time of at least $2\eps$, i.e., the evader moves to the right in Fig.~\ref{fig:fat_triangle} (the case the evader moves in the half-plane $v_{QE}'e(t)\le0$ is similar). Since $\eps\to0$, such assumption is not restrictive.
	The strategy $\widehat\cC$ works as follows.
	\begin{enumerate}
		\item Initially, $P_1$ moves towards $T_1$, while $P_2$ and $P_3$ move symmetrically to the evader w.r.t. $\overline{V_1(0)V_3(0)}$ and $\overline{V_1(0)V_2(0)}$, respectively, see Fig.~\ref{fig:right_triangle_bis}.
		\item As soon as $(P_1-E)$ is parallel to $(V_1-V_2)$, the pursuers play the decentralized strategy $\cD$ until capture occurs. 
	\end{enumerate}
	Notice that in the first step, $P_2$ and $P_3$ move in such a way to guarantee that $V_1$, $v_{12}$ and $v_{13}$ remain the same, where $v_{12}$ and $v_{13}$ are defined as in \eqref{eq:v_i}.
	
	Let $\tau$ be the time at which $\widehat\cC$-strategy switches to step 2 and by $\eta$ the remaining time to conclude the game. Since by \eqref{eq:2eps} $\|P_1(0)-T_1\|\le 2\eps$, one has $\tau\le2\eps$. So, $\cM_{\widehat\cC}=\tau+\eta\le 2\eps+\eta$.
	
	Notice that at time $\tau$ the Voronoi cell becomes a right triangle like the one reported in the proof of Theorem~\ref{th:tight_sx}, see Fig.~\ref{fig:right_triangle_bis}. So, by Theorem~\ref{th:tight_sx}, one has $\lim_{\eps\to 0}\eta=\|V_1(\tau)-V_2(\tau)\|$. Then, as $\eps$ tends to 0, it holds
	\begin{align}
	\lim_{\eps\to0}\cM_{\widehat\cC}&=\lim_{\eps\to0}\tau+\|V_1(\tau)-V_2(\tau)\|\nonumber\\
	&\le \lim_{\eps\to0} 2\eps+\|V_1(\tau)-V_2(\tau)\|\nonumber\\
	&= \lim_{\eps\to0} \|V_1(\tau)-V_2(\tau)\|\nonumber\\
	&\le \|V_1(0)-H\|\nonumber
	\end{align}
	where the last inequality comes from the fact that  $V_1(\tau)=V_1(0)$.
	Thus, by \eqref{eq:MDproof} and by the fact that the optimal cooperative strategy is such that $M_\cC \le M_{\widehat \cC}$, one has
	\begin{align}
	\lim_{\eps\to0}M_\cC &\le \lim_{\eps\to0}M_{\widehat \cC}\le \|V_1(0)-H\|=\frac{1}{2}\|V_2(0)-V_1(0)\|\nonumber\\
	&=\frac{1}{2}\lim_{\eps\to0}M_\cD\ .\nonumber
	\end{align}
	Since by \eqref{eq:main_result}, $M_\cC\ge M_\cD/2$, one gets
	$$
	\lim_{\eps\to0}M_\cC=\frac{1}{2}\lim_{\eps\to0}M_\cD
	$$
	which concludes the proof.
\end{proof}

\begin{remark}
	Notice that the two-step pursuers' strategy $\widehat\cC$ adopted in the proof of Theorem~\ref{th:tight_dx} is cooperative because in the first step $P_1$ moves orthogonally to the largest edge of $\cV(0)$. Since the vertices $\cV(0)$ are defined by the position of all the pursuers (and the evader), it is apparent that such strategy requires complete knowledge of the game state.
\end{remark}

\section{Examples}\label{sec:example}

Let us define $\delta=\frac{M_\cC}{M_\cD}$. By Theorem~\ref{th:main}, one has that in general
$
0.5\le\delta\le 1\ .
$

By Theorem~\ref{th:tight_dx}, one has that there exist games such that $\delta=0.5$, which means that playing a cooperative strategy for the pursuers halves the game duration w.r.t. the decentralized strategy. 

In this section, some examples are reported to show how the range of $\delta$ changes for some representative games. Clearly, the smaller is $\delta$, the larger improvement can be obtained by playing a cooperative strategy.

For a given game, $M_\cD$ and $\underline B$ can be easily computed by \eqref{eq:M_D} and \eqref{eq:LB_general2}, respectively. So, by \eqref{eq:disuguaglianze}, one has
\begin{equation}\label{eq:delta}
\frac{\underline B}{M_\cD}\le\delta\le 1\ .
\end{equation}
Define $\underline\delta=\underline B/M_\cD$. Clearly, $\underline \delta$ provides a lower bound to $\delta$, i.e., the maximum game length reduction which can be obtained by playing the strategy $\cC$ w.r.t. $\cD$.

In Fig.~\ref{fig:examples}, the initial Voronoi cell $\cV(0)$ and the evader position for different games are reported. For the sake of simplicity, the positions of the pursuers are not reported since they can be easily derived from $\cV(0)$ and $E(0)$.
As an example, if $\cV(0)$ is an equilateral triangle of side $l$ and $E$ lies in its geometrical center (see Fig.~\ref{fig:examples}-a), one has
$$
\underline B=\frac{\sqrt{3}}{2}l~,~M_\cD=l~,~\underline\delta =\frac{\sqrt{3}}{2}\ .
$$
Moreover, notice that by \eqref{eq:M_D}, when $\cV(0)$ is an equilateral triangle of side $l$ one has $M_\cD=l$ for any position of the evader inside $\cV$.

In Table~\ref{tab:examples}, the value of $\underline \delta$ for each case depicted in Fig.\ref{fig:examples} is reported. For instance, when the initial game configuration is the isosceles triangle c), the maximum benefit of playing in a cooperative way with respect to the $\cD$-strategy is less than $4\%$.

\begin{figure}[htb]
	\centering %
	\includegraphics[width=0.3\columnwidth]{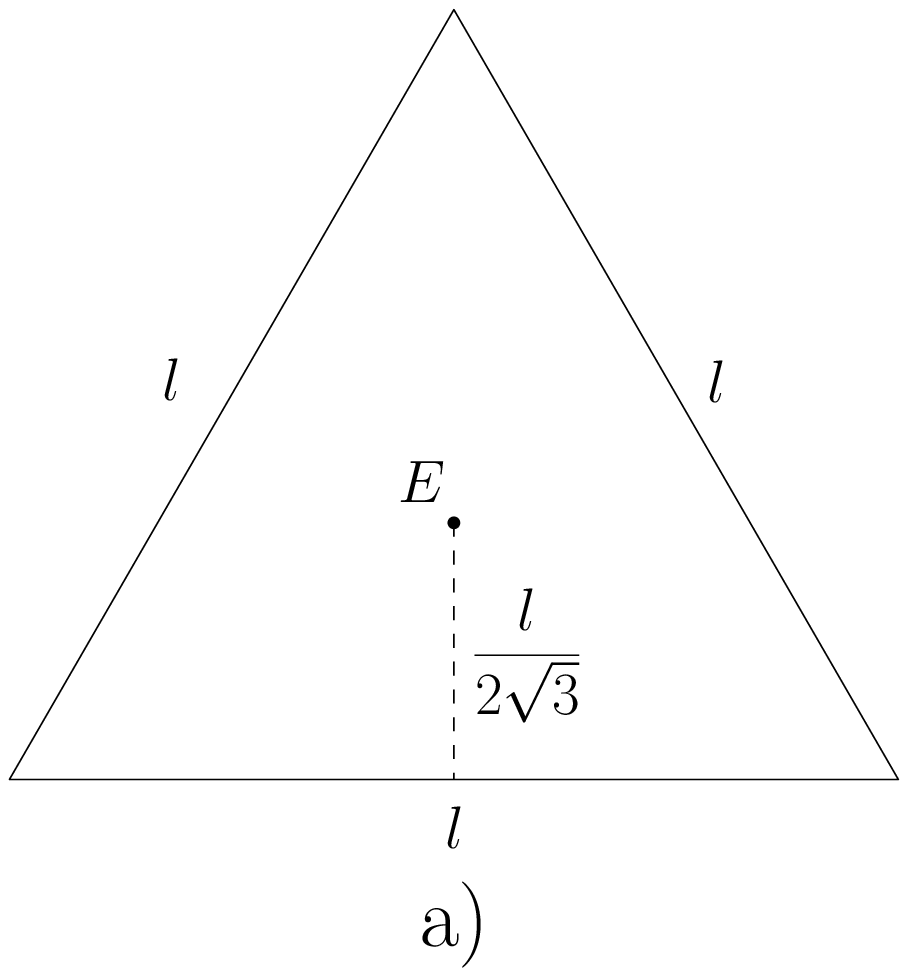}~~~
	\includegraphics[width=0.3\columnwidth]{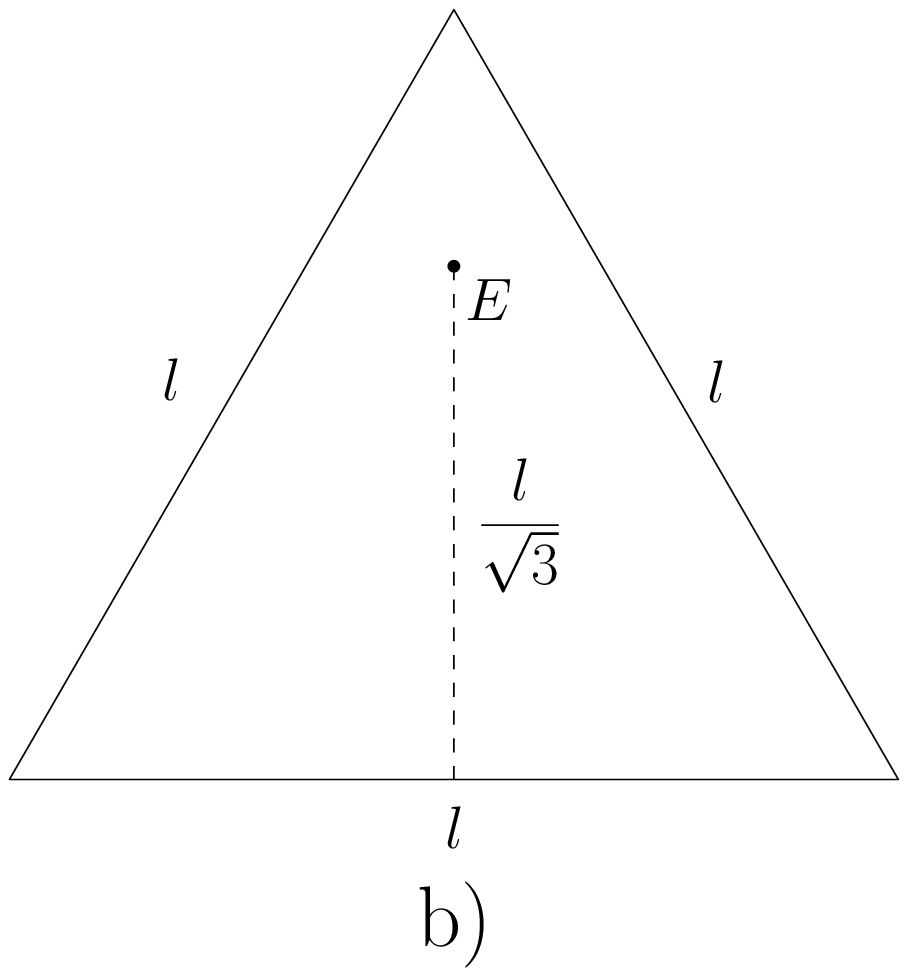}\\[5mm]
	\includegraphics[height=60mm]{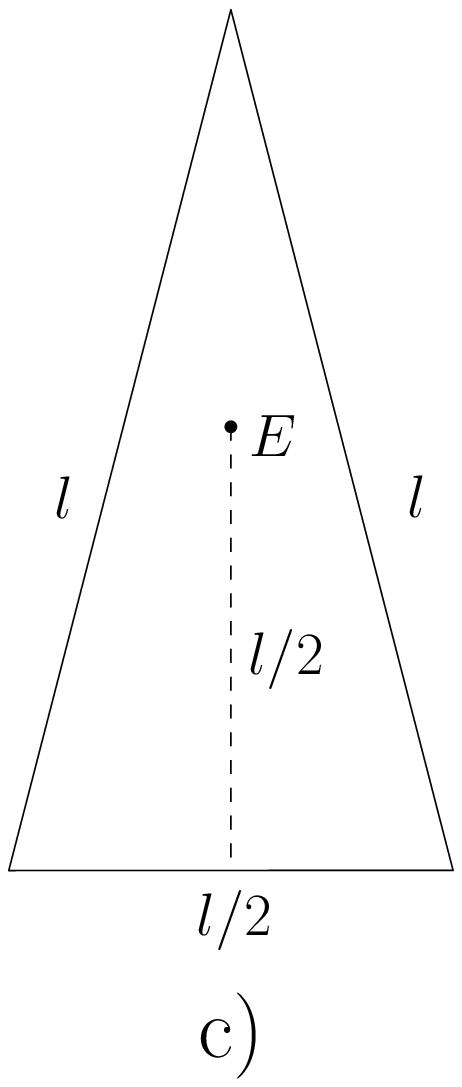}~~~~~~~~~~~~~~~~~
	\includegraphics[height=60mm]{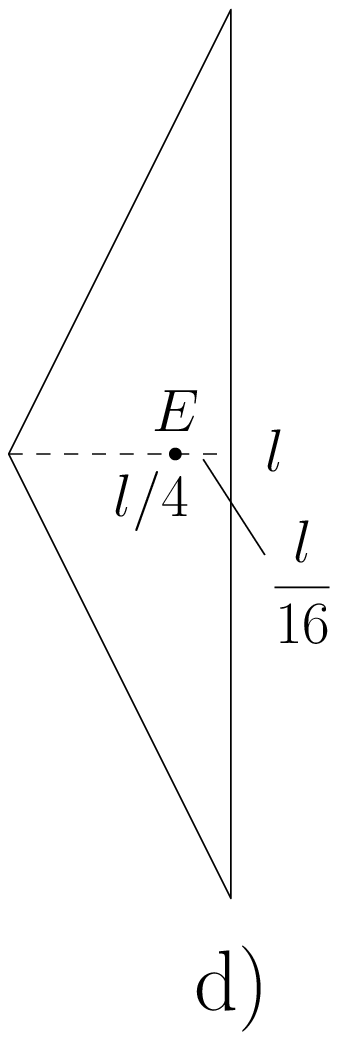}
	\caption{Examples of Voronoi cell $\cV(0)	$.} \label{fig:examples}
\end{figure}

\begin{table}[bt]
	\caption{Ratio $\underline \delta$ for the different games in Fig.\ref{fig:examples}}
	\begin{center}
		\begin{tabular}{|c|c|}
			\hline
			$\cV$ & $\underline \delta$ \\ \hline %
			equilateral triangle a) &   0.866 \\ \hline %
			equilateral triangle b) &   0.902 \\ \hline %
			isosceles triangle c) &   0.968 \\ \hline %
			isosceles triangle d) &   0.690 \\ \hline %
		\end{tabular}
	\end{center}
	\label{tab:examples}
\end{table}

\section{Conclusions}\label{sec:conclusions}
The advantage of cooperation in pursuit-evasion games has been investigated, for a game in the plane with three pursuers and one evader. The maximum reduction of the capture time that can be achieved by using full information on the pursuers' state has been derived. Such reduction has been also specified in terms of the game initial conditions.
This contribution can be seen as a first step towards a deeper understanding of the benefits provided by cooperation of the pursuers. Devising a cooperative pursuit strategy achieving, or at least approaching, the lower bound on the game length, is an open problem. The extension of the results presented in the paper to settings with more than three pursuers or with multiple evaders will be the subject of future research. Another challenging objective is to address multi-pursuer games in bounded environments (e.g., polygons), by building on recent results concerning the single-pursuer case \cite{casini_SCL19,casini_garulli_Automatica19}.

\bibliographystyle{IEEEtran}
\bibliography{lion_and_man}

\begin{thebibliography}{10}
\providecommand{\url}[1]{#1}
\csname url@samestyle\endcsname
\providecommand{\newblock}{\relax}
\providecommand{\bibinfo}[2]{#2}
\providecommand{\BIBentrySTDinterwordspacing}{\spaceskip=0pt\relax}
\providecommand{\BIBentryALTinterwordstretchfactor}{4}
\providecommand{\BIBentryALTinterwordspacing}{\spaceskip=\fontdimen2\font plus
\BIBentryALTinterwordstretchfactor\fontdimen3\font minus
  \fontdimen4\font\relax}
\providecommand{\BIBforeignlanguage}[2]{{%
\expandafter\ifx\csname l@#1\endcsname\relax
\typeout{** WARNING: IEEEtran.bst: No hyphenation pattern has been}%
\typeout{** loaded for the language `#1'. Using the pattern for}%
\typeout{** the default language instead.}%
\else
\language=\csname l@#1\endcsname
\fi
#2}}
\providecommand{\BIBdecl}{\relax}
\BIBdecl

\bibitem{isaacs1965}
R.~Isaacs, \emph{Differential games}.\hskip 1em plus 0.5em minus 0.4em\relax
  New York: Wiley, 1965.

\bibitem{Petrosyan1993}
L.~A. Petrosyan, \emph{Differential games of pursuit}.\hskip 1em plus 0.5em
  minus 0.4em\relax World Scientific, 1993.

\bibitem{chernousko1985}
F.~Chernousko and V.~Zak, ``On differential games of evasion from many
  pursuers,'' \emph{Journal of optimization theory and applications}, vol.~46,
  no.~4, pp. 461--470, 1985.

\bibitem{pashkov1995}
A.~Pashkov and A.~Sinitsyn, ``Construction of the value function in a
  pursuit-evasion game with three pursuers and one evader,'' \emph{Journal of
  Applied Mathematics and Mechanics}, vol.~59, no.~6, pp. 941--949, 1995.

\bibitem{alexander2009}
S.~Alexander, R.~Bishop, and R.~Ghrist, ``Capture pursuit games on unbounded
  domains,'' \emph{Enseign. Math.(2)}, vol.~55, no. 1-2, pp. 103--125, 2009.

\bibitem{vidal2002}
R.~Vidal, O.~Shakernia, H.~J. Kim, D.~H. Shim, and S.~Sastry, ``Probabilistic
  pursuit-evasion games: theory, implementation, and experimental evaluation,''
  \emph{IEEE Transactions on Robotics and Automation}, vol.~18, no.~5, pp.
  662--669, 2002.

\bibitem{bopardikar2008}
S.~D. Bopardikar, F.~Bullo, and J.~P. Hespanha, ``On discrete-time
  pursuit-evasion games with sensing limitations,'' \emph{IEEE Transactions on
  Robotics}, vol.~24, no.~6, pp. 1429--1439, 2008.

\bibitem{chung2011}
T.~H. Chung, G.~A. Hollinger, and V.~Isler, ``Search and pursuit-evasion in
  mobile robotics,'' \emph{Autonomous robots}, vol.~31, no.~4, pp. 299--316,
  2011.

\bibitem{noori2016}
N.~Noori, A.~Beveridge, and V.~Isler, ``Pursuit-evasion: A toolkit to make
  applications more accessible,'' \emph{IEEE Robotics \& Automation Magazine},
  vol.~23, no.~4, pp. 138--149, 2016.

\bibitem{Makkapati2018}
V.~R. Makkapati, W.~Sun, and P.~Tsiotras, ``Optimal evading strategies for
  two-pursuer/one-evader problems,'' \emph{Journal of Guidance, Control, and
  Dynamics}, vol.~41, no.~4, pp. 851--862, 2018.

\bibitem{kopparty2005}
S.~Kopparty and C.~V. Ravishankar, ``A framework for pursuit evasion games in
  $\rr^n$,'' \emph{Information Processing Letters}, vol.~96, no.~3, pp.
  114--122, 2005.

\bibitem{pshenichnyi1976}
B.~Pshenichnyi, ``Simple pursuit by several objects,'' \emph{Cybernetics and
  Systems Analysis}, vol.~12, no.~3, pp. 484--485, 1976.

\bibitem{pshenichnyi1981}
B.~Pshenichnyi, A.~Chikrii, and I.~Rappoport, ``An efficient method of solving
  differential games with many pursuers,'' in \emph{Soviet Mathematics
  Doklady}, vol.~23, no.~1, 1981.

\bibitem{zhou2016}
Z.~Zhou, W.~Zhang, J.~Ding, H.~Huang, D.~M. Stipanovi{\'c}, and C.~J. Tomlin,
  ``Cooperative pursuit with {V}oronoi partitions,'' \emph{Automatica},
  vol.~72, pp. 64--72, 2016.

\bibitem{Ramana2017}
M.~V. Ramana and M.~Kothari, ``Pursuit-evasion games of high speed evader,''
  \emph{Journal of Intelligent {\&} Robotic Systems}, vol.~85, no.~2, pp.
  293--306, Feb 2017.

\bibitem{kothari2017}
M.~Kothari, J.~G. Manathara, and I.~Postlethwaite, ``Cooperative multiple
  pursuers against a single evader,'' \emph{Journal of Intelligent \& Robotic
  Systems}, vol.~86, no. 3-4, pp. 551--567, 2017.

\bibitem{Ramana2017b}
M.~V. Ramana and M.~Kothari, ``Pursuit strategy to capture high-speed evaders
  using multiple pursuers,'' \emph{Journal of Guidance, Control, and Dynamics},
  vol.~40, no.~1, pp. 139--149, 2017.

\bibitem{Jankovic1978}
V.~Jankovi{\'c}, ``About a man and lions,'' \emph{Matemati{\v{c}}ki Vesnik},
  vol.~2, no.~15, pp. 359--362, 1978.

\bibitem{sgall01}
J.~Sgall, ``Solution of {D}avid {G}ale's lion and man problem,''
  \emph{Theoretical Computer Science}, vol. 259, no.~1, pp. 663--670, 2001.

\bibitem{casini_garulli_LCSS17}
M.~Casini and A.~Garulli, ``An improved lion strategy for the lion and man
  problem,'' \emph{IEEE Control Systems Letters (L-CSS)}, vol.~1, no.~1, pp.
  38--43, 2017.

\bibitem{casini_SCL19}
M.~Casini, C.~Matteo, and A.~Garulli, ``A new class of pursuer strategies for
  the discrete-time lion and man problem,'' \emph{Systems \& Control Letters},
  vol. 125, pp. 22--28, 2019.

\bibitem{casini_garulli_Automatica19}
M.~Casini and A.~Garulli, ``A new class of pursuer strategies for the
  discrete-time lion and man problem,'' \emph{Automatica}, vol. 100, pp.
  162--170, 2019.

\end{thebibliography}

\appendix[Proof of Theorem~\ref{th:mosse_decentralizzate}]

In the first step of the $\cE$-strategy, the evader goes from $E(0)$ to $Q$. So, step 1 takes a time $\tau_1=\|E(0)-Q\|$. In step~1, according to \eqref{decentralized_strategy_a}, $P_1$ and $P_3$ move in the same direction of the evader along $e(0)=v_{QE}$, while $P_2$ obeys to \eqref{decentralized_strategy_b} going towards $Q$. As a result, only the smallest edge of $\cV$ moves (along $e(0)$), while the others remain the same. The Voronoi cell at time $\tau_1$ is depicted in red in Fig.~\ref{fig:M_D_steps}.

Since $V_1(\tau_1)=V_1(0)\in\cV(\tau_1)$, in the second step of the $\cE$-strategy the evader moving along $v_{V_1 Q}$ may safely approach $V_1(0)$ at time $\tau_{12}=\tau_1+\tau_2$, with $\tau_2=\|Q-V_1(0)\|$. So, $V_1(\tau_{12})=V_1(0)$. Notice that, as in the previous step, only the smallest edge of $\cV$ is moving in the second step. In Fig.~\ref{fig:M_D_steps}, $\cV(\tau_2)$ is colored in blue.

During the final step, the evader points towards the farthest vertex of $\cV(\tau_2)$ moving along $v_{SE}$. Since the Voronoi cell at any time is a triangle similar to $\cV(0)$, the farthest vertex from $V_1(\tau_2)=V_1(0)$ turns out to be $V_2(\tau_2)$. By defining $\tau_3=\|V_1(0)-V_2(\tau_2)\|$, it is easy to see that $\tau_3=\|S-E(0)\|$ and then the total traveled time is $\tau_1+\tau_2+\tau_3$, which coincides with \eqref{eq:M_D}. At such a time, the Voronoi cell collapses to one point and capture occurs.

It remains to prove that the $\cE$-strategy is optimal when the pursuers play the $\cD$-strategy, i.e., there exists no other evader's strategy guaranteeing a longer survival. 

At a given time $\tau$, according to \eqref{eq:M_D}, let $M_\cD(\tau)$ denote the residual game length if the pursuers and the evader play the $\cD$-strategy and the $\cE$-strategy, respectively, from time $\tau$ onwards. Let the evader move along a direction $\widehat e\colon\|\widehat e\|=1$ for a time $\Delta\tau>0$, i.e., $\dot E(t)=\widehat e$, $\tau\le t\le\tau+\Delta\tau$. Let $M_\cD(\tau+\Delta\tau)$ be the corresponding residual game length at time $\tau+\Delta\tau$. Let us define
\begin{equation}\label{eq:deltaM}
\Delta M(\Delta\tau)=M_\cD(\tau+\Delta\tau)-M_\cD(\tau)
\end{equation}
As it has been shown above, if $\widehat e\in\{v_{QE},\,v_{V_1Q},\,v_{SE}\}$, one has $\Delta M(\Delta\tau)=-\Delta\tau$. Hence, along the three directions the evader follows in the $\cE$-strategy, it holds
$$
\frac{d M_\cD(t)}{dt}=\lim_{\Delta\tau\to0}\frac{\Delta M(\Delta\tau)}{\Delta\tau}=-1\ .
$$
In the following, we prove that for any direction $\widehat e\notin\{v_{QE},\,v_{V_1Q},\,v_{SE}\}$ one has $\frac{d M_\cD(t)}{dt}<-1$. So, the evader will be captured in a shorter time 
and hence any evader's strategy involving a move $\widehat e\notin\{v_{QE},\,v_{V_1Q},\,v_{SE}\}$ cannot be optimal.

Let $\widehat e=[\cos(\theta),\sin(\theta)]'$ with $\theta\in[0,2\pi]$. Assume the evader moves along direction $\widehat e$ for a time $\Delta \tau$. We want to compute $\frac{dM_\cD(t)}{dt}$ as a function of $\theta$. Let $v_{ij}$ be defined as in \eqref{eq:v_i}. Let us consider the six directions $\pm v_{12},\,\pm v_{13},\,\pm v_{23}$, and the resulting six angular intervals in which they partition the interval $[0,2\pi]$, as shown in Fig.~\ref{fig:six_directions}. Let $l=\|V_1-V_2\|$, $m=\|V_1-V_3\|$ and denote by $\varphi_1$ and $\varphi_2$ the angles associated to vertices $V_1$ and $V_2$, respectively.

\begin{figure}[htb]
	\centering %
	\includegraphics[width=0.8\columnwidth]{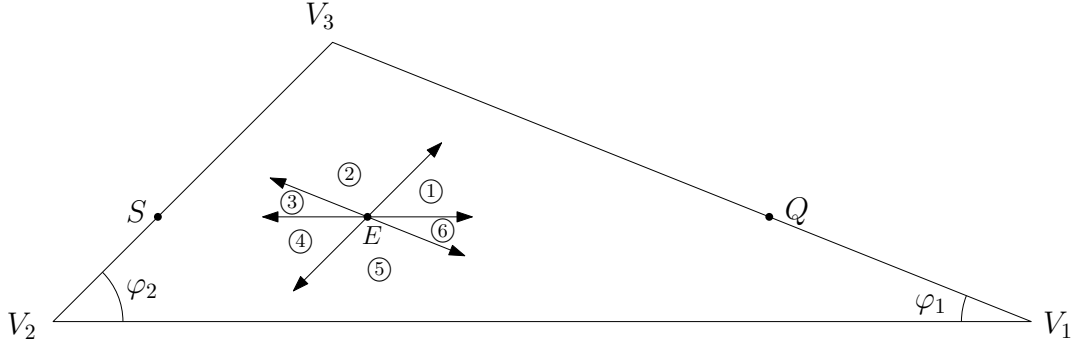}
	\caption{The six directions given by  $\pm v_{12},\,\pm v_{V_13},\,\pm v_{23}$ and the related angular intervals.} \label{fig:six_directions}
\end{figure}

Let us start by assuming $\theta\in[0,\varphi_2)$ and derive the expression of $\Delta M(\Delta\tau)$. Let us refer to Fig.~\ref{fig:evader_move}, where $\cV(\tau)$ and $\cV(\tau+\Delta\tau)$ are depicted in black and red, respectively. By \eqref{eq:M_D}, $M_\cD(\tau)=\|S(\tau)-Q(\tau)\|+\|Q(\tau)-V_1(\tau)\|$. It is easy to see that
$\|Q(\tau+\Delta\tau)-V_1(\tau+\Delta\tau)\|=\|Q(\tau)-V_1(\tau)\|$. 
Let $\widehat Q=Q(\tau)+\Delta\tau\widehat e$ and define $b=\|Q(\tau+\Delta\tau)-\widehat Q\|$, see Fig.~\ref{fig:evader_move}. One has
$$
\|S(\tau+\Delta\tau)-Q(\tau+\Delta\tau)\|=\|S(\tau)-Q(\tau)\|-b
$$
and hence
$$
\Delta M_\cD(\Delta\tau)\!=\!\|S(\tau+\Delta\tau)-Q(\tau+\Delta\tau)\|-\|S(\tau)-Q(\tau)\|\!=\!-b.
$$
\begin{figure}[tb]
	\centering %
	\includegraphics[width=0.8\columnwidth]{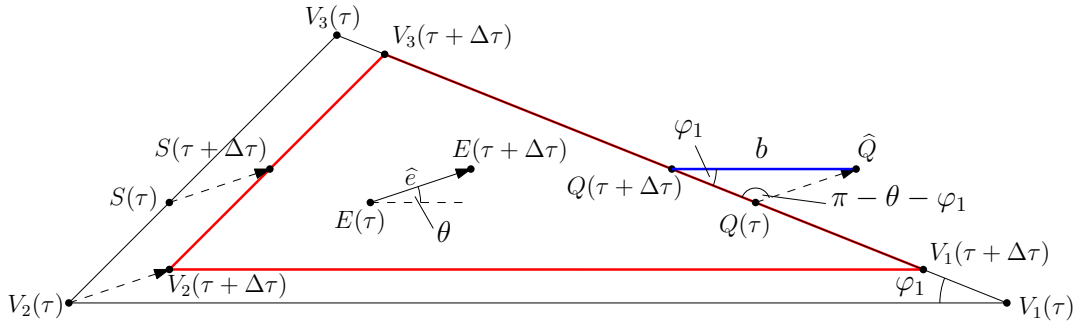}
	\caption{The evader moves along a direction belonging to the first interval. $\cV(\tau)$ and $\cV(\tau+\Delta\tau)$ are depicted in black and red, respectively. The length of the blue segment is equal to $M_\cD(\tau)-M_\cD(\tau+\Delta\tau)$.} \label{fig:evader_move}
\end{figure}
By the law of sines, one has
$$
\frac{b}{\sin(\pi-\theta-\varphi_1)}=\frac{\Delta\tau}{\sin(\varphi_1)}
$$
that is
$$
b=\Delta\tau\frac{\sin(\pi-\theta-\varphi_1)}{\sin(\varphi_1)}=\Delta\tau\frac{\sin(\theta+\varphi_1)}{\sin(\varphi_1)}\ .
$$
Thus, one has
$$
\Delta M_\cD(\Delta\tau)=-\Delta\tau\frac{\sin(\theta+\varphi_1)}{\sin(\varphi_1)}
$$
and hence
$$
\frac{d M_\cD(t)}{dt}=\lim_{\Delta\tau\to0}\frac{\Delta M(\Delta\tau)}{\Delta\tau}= -\frac{\sin(\theta+\varphi_1)}{\sin(\varphi_1)}\ .
$$
By using a similar reasoning, one can compute $dM_\cD(t)/dt$ for all the other cases.  Table~\ref{tab:casistica} reports the expressions of $d M_\cD(t)/dt$ for $\theta$ belonging to the six angular intervals.

By straightforward calculus arguments, it is possible to show that such a function has three maxima in $[0,2\pi)$, all equal to $-1$. As expected, they are achieved when $\theta$ is equal to $0$, $\pi$ and $2\pi-\varphi_1$, which correspond to the directions $v_{QE}$,  $v_{SE}$, $v_{V_1 Q}$ adopted in the $\cE$-strategy. Therefore, any other direction leads to a greater reduction of $M_\cD$ and thus it cannot be optimal.
\vspace*{-18pt}\flushright{$\blacksquare$}

\begin{table}[htb]\setlength\extrarowheight{4pt}
	\caption{Expressions of $d M_\cD(t)/dt$ as a function of $\theta$}
	\label{tab:casistica}
	\begin{center}
		\begin{tabular}{|c|c|c|}
			\hline
			Case & $\theta$ interval & $\frac{d M_\cD(t)}{dt}$ \\[3pt] \hline %
			1 & $[0,\varphi_2)$ 		   	   & $-\frac{\sin(\theta+\varphi_1)}{\sin(\varphi_1)}$ \\ \hline %
			2 & $[\varphi_2,\pi-\varphi_1)$ 	   & $-\frac{\sin(\theta+\varphi_1)}{\sin(\varphi_1)} - \frac{\sin(\theta-\varphi_2)}{\sin(\varphi_2)}  $ \\ \hline %
			3 & $[\pi-\varphi_1,\pi)$    	   & $-\frac{\sin(\theta-\varphi_2)}{\sin(\varphi_2)}$ \\ \hline %
			4 & $[\pi,\pi+\varphi_2)$    	   & $-\frac{\sin(\theta-\varphi_2)}{\sin(\varphi_2)} + \frac{\sin(\theta)}{\sin(\varphi_1)}$ \\ \hline %
			5 & $[\pi+\varphi_2,2\pi-\varphi_1)$ & $\frac{\sin(\theta)}{\sin(\varphi_1)}$ \\ \hline %
			6 & $[2\pi-\varphi_1,2\pi)$  	   & $\frac{\sin(\theta)}{\sin(\varphi_1)} - \frac{\sin(\theta+\varphi_1)}{\sin(\varphi_1)}$ \\ \hline %
		\end{tabular}
	\end{center}
	\label{tab:cases}
\end{table}
\end{document}